\documentclass{amsart}
\usepackage{amsthm}
\usepackage{amssymb}
\usepackage{mathtools}
\usepackage[all]{xy}
\usepackage{tikz-cd}

\newcommand{\cU}{\mathcal{U}}

\newcommand{\Br}{\mathfrak{Br}}
\newcommand{\Hom}{\operatorname{Hom}}
\newcommand{\Bc}{\mathfrak{C}}
\newcommand{\As}{A_\circ}
\newcommand{\AHopf}{(A,\Delta,\epsilon,\cdot,1,S)}
\newcommand{\AsHopf}{(A,\Delta,\epsilon,\circ,1,T)}
\newcommand{\Bs}{B_\circ}
\newcommand{\Mp}{\mathfrak{Mp}}
\newcommand{\Abrace}{(A,\cdot,\circ)}
\newcommand{\Bbrace}{(B,\cdot,\circ)}
\newcommand{\N}{\mathbb{N}}

\newcommand{\GL}{\mathrm{GL}}
\newcommand{\id}{\mathrm{id}}

\newcommand{\g}{\mathfrak{g}}
\newcommand{\h}{\mathfrak{h}}
\newcommand{\sld}{\mathfrak{sl}_2}

\newcount\colveccount
\newcommand*\colvec[1]{
\global\colveccount#1
\begin{pmatrix}
        \colvecnext
        }
        \def\colvecnext#1{
        #1
        \global\advance\colveccount-1
        \ifnum\colveccount>0
        \\
        \expandafter\colvecnext
        \else
\end{pmatrix}
\fi
}

\makeatletter

\numberwithin{equation}{section}
\numberwithin{figure}{section}
\numberwithin{table}{section}

\theoremstyle{plain}
\newtheorem{thm}{Theorem}[section]
\theoremstyle{plain}
\newtheorem{lem}[thm]{Lemma}
\newtheorem{cor}[thm]{Corollary}
\theoremstyle{plain}
\newtheorem{pro}[thm]{Proposition}
\newtheorem*{conjecture*}{Conjecture}

\newtheorem{example}[thm]{Example}
\theoremstyle{remark}
\newtheorem{rem}[thm]{Remark}
\theoremstyle{plain}
\newtheorem{exa}[thm]{Example}
\newtheorem{defn}[thm]{Definition}
\newtheorem{notation}[thm]{Notation}

\makeatother

\newcounter{commentcounter}

\begin{document}

\title{Hopf braces and Yang-Baxter operators}

\begin{abstract}
This paper introduces Hopf braces, a new algebraic structure related to the Yang--Baxter equation which include Rump's braces and their non-commutative generalizations as particular cases. Several results of classical braces are still valid in our context. Furthermore, Hopf braces provide the right setting
for considering left symmetric algebras as Lie-theoretical analogs of braces. 
\end{abstract}

\author{Iv\'an Angiono}
\address{FaMAF-CIEM (CONICET), Universidad Nacional de C\'órdoba, Medina Allende
s/n, Ciudad Universitaria, 5000 C\'órdoba, Argentina}
\email{angiono@famaf.unc.edu.ar}

\author{C\'esar  Galindo}
\address{Departamento de matem\'áticas, Universidad de los Andes, Carrera 1 N. 18A -
10, Bogot\'á, Colombia}
\email{cn.galindo1116@uniandes.edu.co}

\author{Leandro Vendramin}
\address{
Departamento de Matem\'atica -- FCEN,
Universidad de Buenos Aires, Pab. I -- Ciudad Universitaria (1428)
Buenos Aires, Argentina}
\email{lvendramin@dm.uba.ar}

\maketitle


\section*{Introduction}

A \emph{Yang--Baxter operator} on vector space $V$ is an invertible linear endomorphism $c\in\GL(V\otimes V)$ satisfying the \emph{braid  equation} 
\begin{align*}
    (c\otimes\id)(\id\otimes c)(c\otimes\id)=(c\otimes\id)(\id\otimes c)(c\otimes\id).
\end{align*}

Attempts to find solutions of the braid equation turned out to be an important problem that led to the theory of quantum groups. 

Quantum groups 
have remarkable applications in algebra,  low-dimensional topology, differential equations and mathematical physics, see for example~\cite{MR1321145}. These applications are mainly based on the existing connection between quantum groups and Yang--Baxter operators.
Due to the importance of the braid equation equation in mathematics and physics, Drinfeld proposed to study set-theoretical solutions~\cite{MR1183474}. 
A set-theoretical solution of the braid equation is a pair $(X,r)$, 
where $X$ is a set and $r\colon X\times X\to X\times X$ is a bijective
map such that 
\[
(r\times\id)(\id\times r)(r\times\id)=(r\times\id)(\id\times r)(r\times\id).
\]
The first works on set-theoretical solutions are those of Etingof, Schedler and Soloviev~\cite{MR1722951} 
and Gateva-Ivanova and Van den Bergh~\cite{MR1637256}; 
these papers are devoted to involutive solutions. 

There are several
connections between involutive solutions and other branches 
of mathematics. However, 
the structure of set-theoretical solutions is far from being understood. 
To understand the structure behind non-degenerate involutive set-theoretical solutions, in~\cite{MR2278047} Rump introduced \emph{braces}. A (left) brace is an abelian group $(A,+)$ with another group structure, defined via $(a,b)\mapsto ab$, such that the compatibility condition 
\[
a(b + c) + a = ab + ac
\]
holds for all $a, b, c\in A$. The theory of braces is being developed quite intensively, see for example~\cite{BachillerP3,MR3465351,MR3177933,GI15,RumpClassical,2015arXiv150900420S}. One advantage of the language of braces is that one can imitate ring theory to discuss braided groups and sets. Just as in ring theory, one can define right and two-sided ideals of braces, and study their properties. Thus it is a common belief that braces provide the right setting for studying involutive set-theoretical solutions of the Yang--Baxter equation. Applications to ring theory and group theory are also expected. 


The purpose of the this work is to introduce Hopf braces, a new algebraic structure related to the Yang--Baxter equation, which include Rump's braces and their non-commutative generalizations~\cite{GV} as particular cases. 
Our generalization is based on Hopf algebras. Remarkably, 
several results of classical braces are still valid in our context; for example every Hopf brace $H$ produce 
a Yang--Baxter operator on $H$. Furthermore, since Hopf algebras
generalize simultaneously groups and Lie algebras, 
our structure provides the right setting
for considering Lie-theoretical analogs of braces. In this context, left symmetric algebras naturally appear. 

\medskip
The paper is organized as follows. In Section~\ref{braces} we 
define Hopf braces and prove their main properties. In Theorem~\ref{thm:main}
we prove that Hopf braces are equivalent to bijective $1$-cocycles. 
Section~\ref{cocommutative} is devoted to study Hopf braces over
cocommutative Hopf algebras. In Corollary~\ref{cor:YB} 
we prove that Hopf braces over cocommutative Hopf algebras
naturally produce solutions of the Yang--Baxter equation. In Section~\ref{matched} the connection between Hopf braces 
and matched pairs of commutative Hopf algebras is explored, see Theorem~\ref{thm:matched}. 
Section~\ref{LSA} explore the connection between Hopf braces and left symmetric algebras. 
In Section~\ref{schemes} we show that our constructions allow us to consider solutions in the category of affine schemes. 

\section{Hopf braces}
\label{braces}

Our Hopf-theoretical generalization of the concept of a brace is based on the definition given by Ced\'o, Jespers and Okni{\'n}ski, see~\cite[Definition 1]{MR3177933}.

\begin{defn}
	\label{def:brace}
	Let $(A,\Delta,\epsilon)$ be a coalgebra. A \emph{Hopf brace structure} over $A$ consist of the following data:
	\begin{enumerate}
	    \item a Hopf algebra structure $(A,\cdot,1,\Delta,\epsilon,S)$ and 
	    \item a Hopf algebra structure $(A,\circ,1_\circ,\Delta,\epsilon,T)$ 
	\end{enumerate}
	satisfying the following compatibility:
	\begin{align}
		\label{eq:brace}
		&a\circ(bc)=(a_1\circ b)S(a_2)(a_3\circ c),&& a,b,c\in A.
	\end{align}
\end{defn}

\begin{notation}
	Given a Hopf brace as in the Definition \ref{def:brace}, we write $A$ for
	the Hopf algebra structure $(A,\cdot,1,\Delta,\epsilon,S)$ and $\As$ for
	the other one. The Hopf brace is denoted by $\Abrace$. 
\end{notation}

\begin{rem}	
	\label{rem:1=1*}
	In any Hopf brace, $1_\circ=1$. Indeed, 
	setting $a=b=1_\circ$ in \eqref{eq:brace} one obtains $1_\circ c=c$ for all $c$. Similarly, $a=c=1_\circ$ yields $b1_\circ=b$ for all $b$.
\end{rem}

\begin{example}
Recall from~\cite{GV} that a \emph{skew left brace} is a group $A$ with an additional group structure given by $(a,b)\mapsto a\circ b$ such that  $a\circ(bc)=(a\circ b)a^{-1}(a\circ c)$ holds for all $a,b,c\in A$, where $a^{-1}$ denotes the inverse
of $a$ with respect to the group structure given by $(a,b)\mapsto ab$. 

The group algebra and its dual are classical examples of commutative Hopf algebras~\cite[Chapter 2]{MR594432}. Given a brace $A$ and a field $k$, the group algebra $kA$ of $A$ (respectively, $k^A$ the algebra of functions over $A$) yields a cocommutative Hopf brace (respectively, a commutative Hopf co-brace). Thus as a basic example of a Hopf brace, we may take the group algebra of a (classical or skew) brace. 
\end{example}

The following two examples are
based on the semidirect product of cocommutative Hopf algebras, see~\cite[\S2.4]{MR594432}. 

\begin{example}
Let $H$ and $K$ be cocommutative Hopf algebras. Assume that $H$ is a left $K$-module bialgebra. Then $H\#K$ with
\begin{align*}
    &(h\#k)(h'\#k')=hh'\#kk',&&
    (h\#k)\circ (h'\#k')=h(k_1\rightharpoonup h')\# k_2k',
\end{align*}
where $h,h'\in H$ and $k,k'\in K$, is a Hopf brace. 
\end{example}

\begin{example}
Let $H$ and $K$ be cocommutative Hopf algebras. Assume that $K$ is commutative, $H$ is a left $K$-module bialgebra via $\rightharpoonup$ and $k\rightharpoonup (k'\cdot h)=
k'\cdot (k\rightharpoonup h)$ for all $k,k'\in K$ and $h\in $. Then 
\begin{align*}
    &(h\#k)(h'\#k')=h(k_1\cdot h')\# k_2k',\\
    &(h\#k)\circ (h'\#k')=h(k_1\rightharpoonup h')\# k_2k',
\end{align*}
where $h,h'\in H$ and $k,k'\in K$, is a Hopf brace. 
\end{example}

Let $\Abrace$ and $\Bbrace$ be Hopf braces. A \emph{homomorphism} of Hopf
braces $f\colon \Abrace\to\Bbrace$ is a linear map $f$ between (the vector
spaces) $A$ and $B$ such that $f\colon A\to B$  and $f\colon \As\to\Bs$ are
Hopf algebra homomorphisms. Hopf braces form a category.

Fix a Hopf algebra $A=(A,\cdot,1,\Delta,\epsilon,S)$.  Let $\Br(A)$ be the full
subcategory of the category of Hopf braces with objects $\Abrace$. This means that
the objects of $\Br(A)$ are the Hopf braces such that the first Hopf algebra
structure is that of $A$. 

\begin{lem}
	\label{lem:truco}
	Let $\Abrace$ be a Hopf brace. Then
	\[
		S(a_1\circ b)a_2=S(a_1)(a_2\circ S(b))
	\]
	for all $a,b\in A$. 

	\begin{proof}
		Equation~\eqref{eq:brace} implies that
		\begin{equation}
			\label{eq:truco}
			\epsilon(b)a=a\circ (b_1S(b_2))=(a_1\circ b_1)S(a_2)(a_3\circ S(b_2))
		\end{equation}
		holds for all $a,b\in A$.

		Now let $a,b\in A$. Using~\eqref{eq:truco}, 
		\begin{align*}
		   	S(a_1\circ b)a_2 &= S(a_1\circ b_1\epsilon(b_2))a_2 = S(a_1\circ b_1)\epsilon(b_2)a_2\\
			&=S(a_1\circ b_1)(a_2\circ b_2)S(a_3)(a_4\circ S(b_3))\\
			&=\epsilon(a_1\circ b_1)S(a_2)(a_3\circ S(b_2))
			=S(a_1)(a_2\circ S(b)).
		\end{align*}
		This completes the proof.
	\end{proof}
\end{lem}

\begin{lem}
	\label{lem:lambda}
	Let $\Abrace$ be a Hopf brace. Then $A$ is a left $\As$-module-algebra with
	\[
		a\rightharpoonup b=S(a_1)(a_2\circ b),\quad a,b\in A.
	\]
	\begin{proof}
		By Remark~\ref{rem:1=1*}, 
		\begin{align*}
		   a\rightharpoonup 1&=S(a_1)(a_2\circ 1)=S(a_1)a_2=\epsilon(a)1 &\text{for all }&a\in A.
		\end{align*}
		Now Equation~\eqref{eq:brace} implies that 
		\begin{align*}
		a\rightharpoonup (bc)&=S(a_1)(a_2\circ (bc))=S(a_1)(a_2\circ b)S(a_3)(a_4\circ c)
		=(a_1\rightharpoonup b)(a_2\rightharpoonup c)
		\end{align*}
		for all $a,b,c\in A$. 
		Clearly $1\rightharpoonup a=S(1)(1\circ a)=a$
		holds for all $a\in A$. 
		Now using Lemma~\ref{lem:truco} and the Hopf brace structure, one proves that $(a\circ b)\rightharpoonup c=a\rightharpoonup(b\rightharpoonup c)$ for all $a,b,c\in A$.
	\end{proof}
\end{lem}

\begin{rem}
\label{rem:a*b,ab}
It follows from the definition that
\begin{align}
    \label{eq:a*b}&a\circ b=a_1(a_2\rightharpoonup b),\\
    \label{eq:ab}&ab=a_1\circ (T(a_2)\rightharpoonup b)
\end{align}
for all $a,b\in A$.
\end{rem}

\begin{defn}
	Let $H$ and $A$ be Hopf algebras. Assume that 
	$A$ be a $H$-module-algebra. A \emph{bijective
	$1$-cocycle} is a coalgebra isomorphism $\pi\colon H\to A$ such that 
	\begin{align}\label{eq:cocycle-condition}
		&\pi(hk)=\pi(h_1)(h_2\rightharpoonup \pi(k)) & \text{for all }&h,k\in H.
	\end{align}
\end{defn}


\begin{rem}
Any bijective $1$-cocycle $\pi$ satisfies $\pi(1)=1$. Indeed,  
setting $h=k=1$ it follows that $\pi(1)=\pi(1)\pi(1)$. Hence $\pi(1)=1$ since $\pi(1)$ is a group-like element. 
\end{rem}

Let $\pi\colon H\to A$ and $\eta\colon K\to B$ be bijective $1$-cocycles. 
A \emph{homorphism} between these bijective $1$-cocycles is a pair $(f,g)$ of Hopf algebra maps $f\colon H\to K$, $g\colon A\to B$ such that 
\begin{align*}
&\eta f=g \pi,\\
&g(h\rightharpoonup a)=f(h)\rightharpoonup g(a),&&a\in A,\;h\in H.
\end{align*}
Bijective $1$-cocycles form a category. 

Fix a Hopf algebra $A$. 
Let $\Bc(A)$ be the full subcategory of the category of bijective $1$-cocycles with objects $\pi\colon H\to A$. 

\begin{thm}
    \label{thm:main}
    Let $A$ be a Hopf algebra. Then the categories $\Br(A)$ and $\Bc(A)$ are equivalent. 
	
	\begin{proof}
	    We claim that $F\colon\Br(A)\to\Bc(A)$ given by 
	    \begin{align*}
	    &F\Abrace=(\id_A\colon\As\to A), && F(f)=(f,f)\text{ for }f\colon \Abrace\to (A,\cdot,\circledast),
	    \end{align*}
	    is a functor.  
	    We prove that 
		$\pi=\id_A\colon \As\to A$ is a bijective $1$-cocycle. By Lemma \ref{lem:lambda}, $A$ is a $\As$-module-algebra and 
		\begin{align*}
			&\pi(a_1)(a_2\rightharpoonup \pi(b))=a_1S(a_2)(a_3\circ b)=a\circ b=\pi(a\circ b)&\text{for all }a,b\in A.
		\end{align*}
		Now $(f,f)$ is a homomorphism of $1$-cocycles since $f$ is a Hopf algebra homomorphism for both Hopf algebra structures. Hence the claim follows. 
		
		Now we define $G\colon\Bc(A)\to\Br(A)$ as follows. First
	    $G(\pi\colon H\to A)=\Abrace$, where the new multiplication is given by 
		\[
		a\circ b=\pi(\pi^{-1}(a)\pi^{-1}(b)),\quad
		a,b\in A.
		\]
		Let us prove that $\Abrace$ is a Hopf brace. First $(A,\circ,1,\Delta,\epsilon,T)$ is a Hopf algebra for
		$T=\pi S\pi^{-1}$ since $\pi$ is a coalgebra isomorphism. 
		To prove that $\Abrace$ is a Hopf brace let $a,b\in A$. Since $\pi$ is a $1$-cocycle,
		\begin{align*}
			(a_1\circ b)S(a_2)(a_3\circ c)&=\pi(\pi^{-1}(a_1)\pi^{-1}(b))S(a_2)\pi(\pi^{-1}(a_3)\pi^{-1}(c))\\
			&=a_1(\pi^{-1}(a_2)\rightharpoonup b)S(a_3)a_4(\pi^{-1}(a_5)\rightharpoonup c)\\
			&=a_1(\pi^{-1}(a_2)\rightharpoonup b)(\pi^{-1}(a_3)\rightharpoonup c)
			=a_1(\pi^{-1}(a_2)\rightharpoonup (bc)).
		\end{align*}
        Similarly
		\[
		a\circ (bc)=\pi(\pi^{-1}(a)\pi^{-1}(bc))=\pi(\pi^{-1}(a)_1)(\pi^{-1}(a)_2\rightharpoonup (bc)). 
		\]
		Thus $\Abrace$ is a Hopf brace since $\pi$ is a coalgebra homomorphism. 
		
		For $(f,g)$ a morphism between bijectives $1$-cocycles $\pi$ and $\eta$  
		we define $G(f,g)=g$. For  $a,b\in A$ one computes 
		\begin{align*}
		    g(a\circ b)&=g(\pi(\pi^{-1}(a)\pi^{-1}(b))=\eta f(\pi^{-1}(a)\pi^{-1}(b))\\
		    &=\eta (f\pi^{-1}(a)f\pi^{-1}(b))
		    =\eta(\eta^{-1}g(a)\eta^{-1}g(b))=g(a)\circledast g(b).
		\end{align*}
		Thus it follows that $G$ is a functor. 
		
		Clearly $GF=\id_{\Br(A)}$ and $FG\simeq\id_{\Bc(A)}$. 
	\end{proof}
\end{thm}

\section{Cocommutative Hopf Braces}
\label{cocommutative}

\begin{defn}
A Hopf brace $\Abrace$ is said to be \emph{cocommutative} if the underlying coalgebra $(A,\Delta)$ is cocommutative.
\end{defn}

\begin{lem}
\label{lem:leftright}
Let $\Abrace$ be a cocommutative Hopf brace. Then the following hold:
\begin{enumerate}
    \item $A$ is a right $\As$-module-coalgebra via 
    $a\leftharpoonup b=T(a_1\rightharpoonup b_1)\circ a_2\circ b_2$, $a,b\in A.$
    \item $A$ is a left $\As$-module-coalgebra via $\rightharpoonup$.
    \item  $S(a\rightharpoonup b)=a\rightharpoonup S(b)$ for all  $a,b\in A$.
\end{enumerate}
\end{lem}

\begin{proof}
    Clearly $a\leftharpoonup 1=a$ for all $a\in A$. Now let $a,b,c\in A$. Using the cocommutativity of $A$, the fact that $S$ and $T$ are antipodes and Remark \ref{rem:a*b,ab} one obtains:
    \begin{align*}
        (a\leftharpoonup b)\leftharpoonup c &= T\left( (T(a_1\rightharpoonup b_1)\circ a_2\circ b_2)\rightharpoonup c_1\right)\circ T(a_3\rightharpoonup )\circ a_4\circ b_4\circ c_2\\
        &=T\left( (a_1\rightharpoonup b_1)\circ \left( T(a_2\rightharpoonup b_2)\rightharpoonup ((a_3\circ b_3)\rightharpoonup c_1)\right)\right)\circ a_4\circ b_4\circ c_2\\
        &=T\left((a_1\rightharpoonup b_1)\left((a_2\circ b_2)\rightharpoonup c_1\right)\right)\circ a_3\circ b_3\circ c_2\\
        &=T\left(S(a_1)(a_2\circ b_1)S(a_3\circ b_2)(a_4\circ b_3\circ c_1)\right)\circ a_5\circ b_4\circ c_2\\
        &=T\left(S(a_1)(a_2\circ b_1\circ c_1)\right)\circ a_3\circ b_2\circ c_2\\
        &=T\left(a_1\rightharpoonup (b_1\circ c_1)\right)\circ a_2\circ (b_2\circ c_2)=a\leftharpoonup (b\circ c).
    \end{align*}
Now, the first two items follow from the cocommutativity of $A$.

Finally, since $A$ is cocommutative, $S^2=\id_A$. Then Lemma \ref{lem:truco} implies 
\begin{align*}
S(a\rightharpoonup b)&=S\left( S(a_1)(a_2\circ b)\right)=S(a_1\circ b)a_2=S(a_1)(a_2\circ S(b))=a\rightharpoonup S(b)
\end{align*}
for all $a,b\in A$.
\end{proof}

The following is the Hopf-theoretic version of~\cite[Proposition
2.2]{MR1722951},~\cite[Theorem 6]{MR1769723} and~\cite[Theorem
2.3(iv)]{MR1809284}.

\begin{thm}
\label{thm:trenzas}
	Let $\Abrace$ be a cocommutative Hopf brace. Then $c$ and the braiding
	$\sigma\colon A\otimes A\to A\otimes A$ given by $\sigma(a\otimes
	b)=b_1\otimes S(b_2)ab_3$ produce isomorphic representations of the Braid
	group $\mathbb{B}_n$ on $A^{\otimes n}$ for all $n\in\N_{\geq2}$. 
\end{thm}

\begin{proof}
	For $n\geq2$ we define $\gamma_n,\mu_n\colon A^{\otimes}\to A^{\otimes}$ as follows: 
	\begin{align*}
		& \mu_n(a^{(1)}\otimes\cdots\otimes a^{(n)})=a^{(1)}_1\otimes a^{(1)}_2\rightharpoonup a^{(2)}\otimes\cdots\otimes a^{(1)}_n\rightharpoonup a^{(n)},
	\end{align*}
	$\gamma_2=\gamma\colon A\otimes A\to A\otimes A$,  be given by $\gamma(x\otimes
	y)=x_1\otimes (x_2\rightharpoonup y)$ and recursively $\gamma_n=\mu_n(\id_A\otimes \gamma_{n-1})$ for $n>2$. 
	Since $\gamma$ is invertible with
	inverse $\gamma^{-1}(x\otimes y)=x_1\otimes (T(x_2)\rightharpoonup y)$ and 
	the $\mu_n$ are invertible, it follows by induction that the $\gamma_n$ are invertible. 
	A direct calculation using Lemmas \ref{lem:lambda} and~\ref{lem:leftright}(3) and the cocommutativity of $A$ shows that 
	\begin{align}
		\label{eq:mu_sigma}
		&\mu_n\sigma_{i,i+1}=\sigma_{i,i+1}\mu_n&&\text{for all }n\geq2,\,i\in\{2,\dots,n-1\}.
	\end{align}
	We claim that
	\begin{align}
		\label{eq:gamma_c}
		\gamma_n c_{i,i+1}=\sigma_{i,i+1}\gamma_n&&\text{for all }n\geq2,\,i\in\{1,\dots,n-1\}.
	\end{align}

	We proceed by induction on $n$. Assume $n=2$. 
    We claim that $\gamma c=\sigma\gamma$. 
    Using that $A$ is cocommutative, $S^2=\id_A$ and third  part of Lemma \ref{lem:leftright},
    \begin{align*}
        \gamma c\gamma^{-1}(x\otimes y) &= \gamma c(x_1\otimes (T(x_2)\rightharpoonup y))\\
        &=\gamma\left( S(x_1)(x_2\circ (T(x_3)\rightharpoonup y_1))\otimes  T(y_2)\circ x_4\circ (T(x_5)\rightharpoonup y_3)\right)\\
        &=\gamma(S(x_1)x_2y_1\otimes T(y_2)\circ (xy_3))
        =\gamma(y_1\otimes (T(y_2)\circ x)ST(y_3))\\
        &=y_1\otimes S(y_2)(y_3\circ ((T(y_4)\circ x)ST(y_5)))\\
        &=y_1\otimes S(y_2)xS(y_3)(y_4\circ ST(y_5))
        =y_1\otimes S(y_2)x(y_3\rightharpoonup ST(y_4))\\
        &=y_1\otimes S(y_2)xS(y_3\rightharpoonup T(y_4))
        =y_1\otimes S(y_2)xy_3.
    \end{align*}
	Assume now that the claim holds for $n\geq2$.  
	The case $i=1$ follows immediately from $\gamma c=\sigma\gamma$ since 
	\begin{align*}
		\gamma_n(a^{(1)}\otimes \cdots\otimes a^{(n)})=a^{(1)}_1\otimes a^{(1)}_2\rightharpoonup a^{(2)}_1\otimes\cdots\otimes (a^{(1)}_n\circ\cdots\circ a^{(n-1)}_2)\rightharpoonup a^{(n)}.
	\end{align*}
	For $i>1$ the inductive hypothesis and \eqref{eq:mu_sigma} yield 
	\begin{align*}
		\gamma_{n+1} c_{i,i+1}=\mu_{n+1}(\id_A\otimes \gamma_{n})c_{i,i+1}=\mu_{n+1}\sigma_{i,i+1}(\id_A\otimes \gamma_{n})=\sigma_{i,i+1}\gamma_{n+1}.
	\end{align*}
	This completes the proof of \eqref{eq:gamma_c} and hence the claim follows.
\end{proof}

We now prove that cocommutative Hopf braces produce Yang--Baxter operators.

\begin{cor}
	\label{cor:YB}
	Let $\Abrace$ be cocommutative Hopf brace. Then $c\colon A\otimes A\to A\otimes A$,
	\begin{align*} 
		&c(x\otimes y)=(x_1\rightharpoonup y_1)\otimes (x_2\leftharpoonup y_2),&& x,y\in A,
	\end{align*}
	is a coalgebra isomorphism and a solution of the braid equation.
\end{cor}

\begin{proof}
Theorem~\ref{thm:trenzas} implies that $c$ is an invertible solution of the braid equation. Since both actions are coalgebra maps, $c$ 
is a coalgebra isomorphism.
\end{proof}

The following corollary generalizes~\cite[Proposition 4]{MR1769723}. 

\begin{cor}
	Let $\Abrace$ be a cocommutative Hopf brace. Then $c^2=\id$ if and
	only if $A$ is commutative.
\end{cor}

\begin{proof}
	It is enough to prove that $\sigma^2=\id$ if and only if $A$ is commutative.
	If $\sigma^2=\id$, then applying $\id\otimes\epsilon$ one obtains
	$x\epsilon(y)=S(y_1)xy_2$ for all $x,y\in A$. Then $\sigma(x\otimes y)=y\otimes
	x$ and hence $xy=yx$ since $xy=(m\sigma)(x\otimes y)$. The converse is
	clear.
\end{proof}

\section{Matched pairs of cocommutative Hopf algebras}
\label{matched}
Let $H$ and $K$ be two commutative Hopf algebras.  Recall from~\cite{MR1321145}
that a \emph{matched pair} of cocommutative Hopf algebras is a pair $(H,K)$
with two actions \[
K\xleftarrow{\leftharpoonup}K\otimes
H\xrightarrow{\rightharpoonup}H
\]
such that $(H,\rightharpoonup)$ is a left
$K$-module coalgebra, $(K,\leftharpoonup)$ is a right $H$-module coalgebra, and 
\begin{align}
	\label{eq:matched1}&x\rightharpoonup (ab)=(x_1\rightharpoonup a_1)\left( (x_2\leftharpoonup a_2)\rightharpoonup b \right),\\
	\label{eq:matched2}&(xy)\leftharpoonup a=(x\leftharpoonup (y_1\rightharpoonup a_1))(y_2\leftharpoonup a_2)
\end{align}
for all $a,b\in H$ and $x,y\in K$. 

If $(H,K)$ is a matched pair of cocommutative Hopf algebras, then the tensor
coalgebra $H\otimes K$ is a Hopf algebra with multiplication
\begin{align*}
	(a\otimes x)(b\otimes y)=a(x_1\rightharpoonup b_1)\otimes (x_2\leftharpoonup b_2)y,&& a,b\in H,\,x,y\in K.
\end{align*}

We now show that there is a correspondence between Hopf braces and certain 
matched pairs of cocommutative Hopf algebras. 
A similar result was proved by Gateva-Ivanova for classical braces, see~\cite[Theorem 3.7]{GI15}.  

\begin{pro}
	\label{pro:brace=>matched}
	Let $\Abrace$ be a cocommutative Hopf brace. Then
	$(\As,\As)$ is a matched pair of cocommutative Hopf algebras with 
	\begin{align*}
		&h\rightharpoonup k=S(h_1)(h_2\circ k),&&
		h\leftharpoonup k=T(h_1\rightharpoonup k_1)\circ h_2\circ k_2,&& h,k\in H. 
	\end{align*}
\end{pro}

\begin{proof}
	By Lemma~\ref{lem:leftright}, we need to prove~\eqref{eq:matched1}
	and~\eqref{eq:matched2}.  Let $h,x,y\in H$. Using Remark~\ref{rem:a*b,ab}
	and the cocommutativity,
	\begin{align*}
		(h_1\rightharpoonup x_1)&\circ \left( (h_2\leftharpoonup x_2)\rightharpoonup y \right)
		=(h_1\rightharpoonup x_1)( (h_2\rightharpoonup x_2)\rightharpoonup ( (h_3\leftharpoonup x_3)\rightharpoonup y)\\
		&=(h_1\rightharpoonup x_1)( ( (h_2\rightharpoonup x_2)\circ (h_3\leftharpoonup x_3))\rightharpoonup y)
		=(h_1\rightharpoonup x_1)(h_2\circ x_2\rightharpoonup y)\\
		&=(h_1\rightharpoonup x_1)(h_2\rightharpoonup (x_2\rightharpoonup y))
		=h\rightharpoonup (x_1(x_2\rightharpoonup y))
		=h\rightharpoonup (x\circ y).
	\end{align*}
	Similarly,
	\begin{align*}
		(x\leftharpoonup (y_1\rightharpoonup h_1))\circ (y_2\leftharpoonup h_2) 
		&=T(x_1\circ y_1\rightharpoonup h_1)\circ x_2\circ y_2\circ h_2
		=(x\circ y)\leftharpoonup h. 
	\end{align*}
	This completes the proof.
\end{proof}



Now we prove that matched pairs produce Hopf braces. 

\begin{pro}
	\label{pro:matched=>brace}
	Let $\AsHopf$ be a cocommutative Hopf algebra. Assume that $(A,A)$ 
	is a matched pair of cocommutative Hopf algebras with actions
	$\rightharpoonup$ and $\leftharpoonup$ and that $a\circ
	b=(a_1\rightharpoonup b_1)\circ (a_2\leftharpoonup b_2)$ for all $a,b\in
	A$.  Then $\Abrace$ is a cocommutative Hopf brace with
	\begin{align*}
		&ab=a_1\circ (T(a_2)\rightharpoonup b),&& S(a)=a_1\rightharpoonup T(a_2),&&a,b\in A.
	\end{align*}
\end{pro}

\begin{proof}
	We first notice that since $\rightharpoonup$ is of coalgebras it follows that 
	$\Delta(ab)=\Delta(a)\Delta(b)$ for all $a,b\in A$. Further, 
	\begin{equation}
	\label{eq:matched_a*b}
	\begin{aligned}
		a_1(a_2\rightharpoonup b)&=a_1\circ (T(a_2)\rightharpoonup (a_3\rightharpoonup b))\\
		&=a_1\circ ((T(a_2)\circ a_3)\rightharpoonup b)=a\circ (1\circ b)=a\circ b.
	\end{aligned}
	\end{equation}
	for all $a,b\in A$ 

	Let $a,b,c\in A$. Using the cocommutativity,~\eqref{eq:matched1} and~\eqref{eq:matched_a*b}, 
	\begin{align*}
		a\rightharpoonup (bc)&=a\rightharpoonup (b_1\circ (T(b_2)\rightharpoonup c))
		=(a_1\rightharpoonup b_1)\circ \left( (a_2\leftharpoonup b_2)\rightharpoonup (T(b_3)\rightharpoonup c \right))\\
		&=(a_1\rightharpoonup b_1)\left((a_2\rightharpoonup b_2)\rightharpoonup \left( (a_3\leftharpoonup b_3)\rightharpoonup (T(b_4)\rightharpoonup c) \right)\right)\\
		&=(a_1\rightharpoonup b_1)\left(((a_2\rightharpoonup b_2)\circ (a_3\leftharpoonup b_3))\rightharpoonup (T(b_4)\rightharpoonup c) \right)\\
		&=(a_1\rightharpoonup b_1)\left((a_2\circ b_2)\rightharpoonup (T(b_4)\rightharpoonup c)\right)\\
		&=(a_1\rightharpoonup b_1)\left(a_2\rightharpoonup (b_2\rightharpoonup (T(b_3)\rightharpoonup c))\right)\\
		&=(a_1\rightharpoonup b_1)\left(a_2\rightharpoonup ((b_2\circ T(b_3))\rightharpoonup c)\right)		=(a_1\rightharpoonup b)(a_2\rightharpoonup c).
	\end{align*}
	Now $a(bc)=(ab)c$ since 
	\begin{align*}
		a(bc)&=a_1\circ ( T(a_2)\rightharpoonup (bc))
		=a_1\circ ((T(a_2)\rightharpoonup b)(T(a_3)\rightharpoonup c))\\
		&=a_1\circ (T(a_2)\rightharpoonup b_1)\circ (T(T(a_3)\rightharpoonup b_2)\rightharpoonup (T(a_4)\rightharpoonup c))\\
		&=(a_1b_1)\circ ((T(T(a_3)\rightharpoonup b_2)\circ T(a_4))\rightharpoonup c)\\
		&=(a_1b_1)\circ (T(a_4\circ (T(a_3)\rightharpoonup b_2))\rightharpoonup c)
		=(a_1b_1)\circ (T(a_2b_2)\rightharpoonup c)=(ab)c.
	\end{align*}
	For $a\in A$ let $S(a)=a_1\rightharpoonup T(a_2)$. Since 
	\begin{align*}
		a_1S(a_2) &= a_1(a_2\rightharpoonup T(a_3))=a_1\circ (T(a_2)\rightharpoonup (a_3\rightharpoonup T(a_4))\\
		&=a_1\circ ( (T(a_2)\circ a_3)\rightharpoonup T(a_4))=a_1\circ T(a_2)=\epsilon(a)1
	\end{align*}
	for all $A\in A$, the cocommutativity and~\cite[Theorem 3(4)]{MR585730} imply 
	that the tuple  $\AHopf$ is a Hopf algebra. Now
	\begin{align*}
		&S(a_1)(a_2\circ b)=S(a_1)a_2(a_3\rightharpoonup b)=a\rightharpoonup b&\text{for all }a,b\in A.
	\end{align*}
	This equation implies that $\Abrace$ is a Hopf brace since
	\begin{multline*}
		(a_1\circ b)S(a_2)(a_3\circ c) = (a_1\circ b)(a_2\rightharpoonup c)=(a_1\circ b_1)\circ (T(a_2\circ b_2)\rightharpoonup (a_3\rightharpoonup c))\\
		=(a_1\circ b_1)\circ ((T(b_2)\circ T(a_2)\circ a_3)\rightharpoonup c)
		=a\circ (b_1\circ (T(b_2)\rightharpoonup c))=a\circ (bc)
	\end{multline*}
	for all $a,b,c\in A$. 
\end{proof}

Now, we will show that the correspondence of Propositions~\ref{pro:brace=>matched} and~\ref{pro:matched=>brace} is functorial. 
Let $\AHopf$ be a cocommutative Hopf algebra. Let $\Mp(A)$ be the category with objects the matched pairs $(A,A)$ such that $a\circ b=(a_1\rightharpoonup
b_1)\circ (a_2\leftharpoonup b_2)$ for all $a,b\in A$ and morphisms all 
Hopf algebra homomorphism $f\colon A\to A$ such that $f(a\rightharpoonup b)=f(a)\rightharpoonup f(b)$ $f(a\leftharpoonup b)=f(a)\leftharpoonup f(b)$ for all $a,b\in A$.

\begin{thm}
\label{thm:matched}
    Let $\AHopf$ be a cocommutative Hopf algebra. 
	The categories $\Br(A)$ and $\Mp(A)$ are equivalent. 
\end{thm}

\begin{proof}
    Let $F\colon\Br(A)\to\Mp(A)$ be given by
    $F(\Abrace)=(A,A)$, where $(A,A)$ is the matched pair of Proposition~\ref{pro:brace=>matched}, and $F(f)=f$ for any  morphism $f$ of $\Br(A)$. 
    Then clearly $F$ is a functor. Conversely, let $G\colon\Mp(A)\to\Br(A)$ be given by
    $G(A,A)=\Abrace$, where $\Abrace$ is the Hopf brace of Proposition~\ref{pro:matched=>brace},  
    and $G(f)=f$ for any morphism $f$ of $\Mp(A)$. Since such a morphism $f$ 
    satisfies $f(ab)=f(a)f(b)$ for all $a,b\in A$, it follows that $G$ is a functor.
    Now a direct calculation shows that $\Br(A)$ and $\Mp(A)$ are equivalent. 
\end{proof}

\section{Braces and left symmetric algebras}
\label{LSA}

Left symmetric algebras are non-associative algebras that arise in many areas of mathematics. They were first introduced by Cayley in 1896 in his study of rooted tree algebras and 
later rediscovered by Vinberg and Koszul 
to study convex homogeneous cones and affine flat manifolds. 
Left symmetric algebras are also known as Pre-Lie algebras,
Vinberg algebras, Koszul algebras, quasi-associative algebras and 
Gerstenhaber algebras. We refer to~\cite{MR2233854} for a  survey on left symmetric algebras and their applications. 

Recall that a \emph{left symmetric algebra} is a vector space $V$ with a bilinear map $V\times V\to V$, $(x,y)\mapsto xy$, such that $x(yz)-(xy)z=y(xz)-(yx)z$ for all $x,y,z\in V$. If $V$ is a left symmetric algebra, then $V$ with
$[x,y]=xy-yx$ is a Lie algebra. This Lie algebra will be denoted by $\mathfrak{g}(V)$.

Let $\mathfrak{g}$ be a Lie algebra and $\rho\colon\mathfrak{g}\to\mathfrak{gl}(V)$ be a representation. A $1$-cocycle $\pi$ associated with $\rho$ is a linear map 
$\pi\colon\mathfrak{g}\to V$ such that 
\begin{align}\label{eq:lie-1-cocycle-classical}
    \pi([x,y])&=\rho(x)\pi(y)-\rho(y)\pi(x), & x,y &\in\mathfrak{g}.
\end{align}
It was proved in 
\cite{MR868976,MR654637} that
left symmetric algebras are equivalent to bijective $1$-cocycles, 
see also~\cite[Theorem 2.1]{MR2503192} and~\cite[Proposition 9.1]{RumpClassical}. 
The correspondence goes as follows. 
If $\pi$ is a bijective $1$-cocycle on a Lie algebra $\mathfrak{g}$ with respect to a representation $\rho$, then $x*y=\pi^{-1}(\rho(x)\pi(y))$ defines a left symmetric algebra on $\mathfrak{g}$. Conversely, if $V$ is a left symmetric algebra, then $\id\colon \mathfrak{g}(V)\to V$ is a bijective $1$-cocycle, where $V$ is considered as the $\mathfrak{g}(V)$-module with action given by $L\colon V\to V$, $x\mapsto L_x$, $L_x\colon y\mapsto xy$.


In order to study Hopf braces and bijective 1-cocycles for enveloping algebras we recall the general definition of 1-cocycles of Lie algebras.

\begin{defn}
Let $\g$, $\h$ be Lie algebras and $\rho\colon\g\to\operatorname{Der} \h$ be a Lie algebra map. A \emph{bijective $1$-cocycle} $\pi$ associated with $\rho$ is a linear isomorphism $\pi\colon\g\to\h$ such that 
\begin{align}\label{eq:lie-1-cocycle-general}
    \pi([x,y])&=[\pi(x),\pi(y)]+\rho(x)\pi(y)-\rho(y)\pi(x), & x,y &\in\mathfrak{g}.
\end{align}
\end{defn}

Thus \eqref{eq:lie-1-cocycle-classical} corresponds to \eqref{eq:lie-1-cocycle-general} when $\h=V$ is an abelian Lie algebra.

Recall that the subspace $\mathcal{P}(H)$ of primitive elements of a Hopf algebra $H$ is a Lie algebra with the usual commutator $[x,y]=xy-yx$.

\begin{lem}\label{lema:1-cocycle-from-enveloping-to-Lie}
Let $K$ and $A$ be cocommutative Hopf algebras, and $\pi\colon K\to A$ be a bijective	$1$-cocycle. Let $\g=\mathcal{P}(K)$, $\h=\mathcal{P}(A)$. 
Then $\pi$ restricts to a bijective $1$-cocycle $\pi_{|\g}\colon\g\to\h$.
\end{lem}

\begin{proof}
First, $\pi_{|\g}\colon\g\to\h$ is a linear isomorphism since $\pi$ is a coalgebra isomorphism.
By Lemma \ref{lem:leftright}, $A$ is a left $K$-module-coalgebra, so for each $x\in\g$, $a\in\h$,
\begin{align*}
    \Delta(x\rightharpoonup a) &= x_1\rightharpoonup a_1 \otimes x_2\rightharpoonup a_2 
    = x\rightharpoonup a \otimes 1 + 1\otimes x\rightharpoonup a.
\end{align*}
That is, $x\rightharpoonup a\in\h$ for all $x\in\g$, $a\in\h$. Let $\rho\colon\g\to\operatorname{End} \h$, $\rho(x)(a)=x\rightharpoonup a$. As $\rightharpoonup$ is an action, $\rho$ is a Lie algebra map, 
$\rho(\g)\subseteq \operatorname{Der} \h$. From \eqref{eq:cocycle-condition},
\begin{align*}
 \pi(xy)&=\pi(x)\pi(y)+\rho(x)\pi(y) & \text{for all }&x,y\in \g.
\end{align*}
Thus $\pi$ satisfies \eqref{eq:lie-1-cocycle-general}.
\end{proof}
	
Reciprocally, we can extend a bijective 1-cocycle from Lie algebras to their enveloping algebras.
To prove this we need an auxiliar result.

\begin{lem}\label{lem:extension-lie-algebra-action}
Let $\rho\colon\g\to\operatorname{Der} \h$ be a Lie algebra map. Then $\rho$ extends to an action $\rightharpoonup$ of $\cU(\g)$ on $\cU(\h)$ such that $\cU(\h)$ is a $\cU(\g)$-module-algebra.
Furthermore, under this action $\cU(\h)$ is also a $\cU(\g)$-module-coalgebra.
\end{lem}
\begin{proof}
Indeed $\rho$ extends to a Lie algebra map $\rho\colon\g\to \operatorname{End}\cU(\h)$ such that
\begin{align*}
    \rho(x)(yz)&=\rho(x)(y)z+y\rho(x)(z) & \mbox{for all }&x\in\g, \, y,z \in\cU(\h).
\end{align*}
The corresponding algebra map $\cU(\g)\to \operatorname{End}\cU(\h)$ gives the action $\rightharpoonup$.
We recall that $\cU(\h)$ admits a canonical filtration $(\cU(\h)_n)_{n\in\N_0}$, where $\cU(\h)_n$ is spanned by the product of at most $n$ elements of $\h$. This action preserves the filtration.

To prove the last statement we claim first that
\begin{align}\label{eq:Uh-modulo-coalgebra-claim}
\Delta(x\rightharpoonup y) &= x\rightharpoonup y_1\otimes y_2 + y_1\otimes x\rightharpoonup y_2, & \mbox{for all }&x\in\g, \, y\in\cU(\h).
\end{align}
 We prove by induction on $n$ that \eqref{eq:Uh-modulo-coalgebra-claim} holds for all $ y\in\cU(\h)_n$. The cases $n=0,1$ follows directly since $\cU(\h)_0=k 1$ and $\cU(\h)_1=k 1\oplus \h$. Now assume that \eqref{eq:Uh-modulo-coalgebra-claim} holds for all $y\in\cU(\h)_n$. For each $h\in\h$,
\begin{align*}
\Delta(x\rightharpoonup hy) &= \Delta \big( (x\rightharpoonup h) y+ h(x\rightharpoonup y) \big) 
= x\rightharpoonup (hy)_1\otimes (hy)_2 + (hy)_1\otimes x\rightharpoonup (hy)_2,
\end{align*}
and the inductive step follows since $\cU(\h)_{n+1}-\cU(\h)_n$ is spanned by the elements $hy$, $y\in\cU(\h)_n$, $h\in\h$.

From \eqref{eq:Uh-modulo-coalgebra-claim} we prove that $\Delta(x\rightharpoonup y) = x\rightharpoonup \Delta(y)$ for all $x\in\cU(\g)$, $y\in\cU(\h)$ since $\cU(\g)$ is generated as an algebra by $\g$.
\end{proof}

\begin{lem}\label{lema:1-cocycle-from-Lie-to-enveloping}
Let $\g$, $\h$ be Lie algebras and $\pi\colon\g\to\h$ a bijective $1$-cocycle. Then $\pi$ admits an extension $\pi\colon\cU(\g)\to\cU(\h)$ such that it is a bijective $1$-cocycle.
\end{lem}
\begin{proof}
We fix a basis $(v_i)_{i\in I}$ of $\g$. Let $w_i=\pi(v_i)$, so $(w_i)_{i\in I}$ is a basis of $\h$. Let $\pi\colon\cU(\g)\to\cU(\h)$ be the linear map defined recursively on the PBW basis of $\g$ as follows: $\pi(1)=1$, $\pi(v_i)=w_i$ for all $i\in I$, and for $n\ge 2$, $i_1,\dots,i_n\in I$,
\begin{align}\label{eq:extension-1-cocycle}
\pi(v_{i_1}\dots v_{i_n}) &= w_{i_1}\pi(v_{i_2}\dots v_{i_n}) +
v_{i_1} \rightharpoonup \pi(v_{i_2}\dots v_{i_n}).
\end{align}
Recursively we prove that $\pi(v_{i_1}\dots v_{i_n}) \in  w_{i_1}\dots w_{i_n}+\cU(\h)_{n-1}$, so the matrix of $\pi$ with respect to the PBW bases of $\g$ and $\h$ with generators $(v_i)_{i\in I}$ and $(w_i)_{i\in I}$, respectively, is upper triangular; thus $\pi$ is a linear isomorphism. From \eqref{eq:extension-1-cocycle},
\begin{align}\label{eq:extension-1-cocycle-x-in-g}
\pi(xy) &= \pi(x) \pi(y) + x \rightharpoonup \pi(y) & &\mbox{for all } x\in\g, \, y\in\cU(\g).
\end{align}
We check that \eqref{eq:cocycle-condition} holds for all $x\in\cU(\g)_n$,  $y\in\cU(\g)$ by induction on $n$: the case $n=0$ is direct and the case $n=1$ is \eqref{eq:extension-1-cocycle-x-in-g}. For the inductive step, it is enough to consider products $vx$, $v\in\g$, $x\in\cU(\g)_n$, since these elements span $\cU(\g)_{n+1}$. Using \eqref{eq:extension-1-cocycle-x-in-g} and inductive hypothesis,
\begin{align*}
\pi\big( (vx) y \big) &= \pi(v) \pi(xy) + v \rightharpoonup \pi(xy) \\
&= \pi(v) \pi(x_1) (x_2 \rightharpoonup \pi( y)) + v \rightharpoonup \big( \pi(x_1) (x_2 \rightharpoonup \pi( y)) \big) \\
&= \big( \pi(v) \pi(x_1) + (v \rightharpoonup  \pi(x_1) \big) (x_2 \rightharpoonup \pi( y))+ \pi(x_1) \big( (v x_2) \rightharpoonup \pi( y) \big) \\
&= \pi(vx_1) (x_2 \rightharpoonup \pi( y))+ \pi(x_1) \big( (v x_2) \rightharpoonup \pi( y) \big) = \pi((vx)_1) \big( (vx)_2 \rightharpoonup \pi( y) \big).
\end{align*}
Thus $\pi$ satisfies \eqref{eq:cocycle-condition}. Finally we prove that $\pi$ is a coalgebra map. We claim that $\pi_{|\cU(\g)_n}$ is so, by induction on $n$: the cases $n=0,1$ are direct. Let $v\in\g$, $x\in\cU(\g)_n$:
\begin{align*}
    \Delta \pi (vx) &= \Delta \big( \pi(v) \pi(x) + v \rightharpoonup \pi(x) \big) = (\pi(v)\otimes 1 + 1\otimes \pi(v)) \Delta  \pi(x) + \Delta(v \rightharpoonup \pi(x) ) \\
    &= \pi(vx_1)\otimes \pi(x_2) + \pi(x_1) \otimes \pi(vx_2) = (\pi\otimes\pi) \Delta(vx),
\end{align*}
by \eqref{eq:extension-1-cocycle-x-in-g}, inductive hypothesis and Lemma \ref{lem:extension-lie-algebra-action}. Thus the inductive step follows and $\pi$ is a coalgebra map.
\end{proof}

\begin{pro}
Assume that $k$ is an algebraically closed field of characteristic 0. For each pair of Lie algebras $\g$, $\h$, there exists a bijective correspondence between 
\begin{enumerate}
    \item bijective $1$-cocycles between the Lie algebras $\g$ and $\h$, and
    \item bijective $1$-cocycles between the Hopf algebras $\cU(\g)$ and $\cU(\h)$.
\end{enumerate}
\end{pro}

\begin{proof}
By Cartier--Kostant Theorem, $\mathcal{P}(\cU(\g))=\g$ and $\mathcal{P}(\cU(\h))=\h$. Then we apply Lemmas \ref{lema:1-cocycle-from-enveloping-to-Lie} and \ref{lema:1-cocycle-from-Lie-to-enveloping}.
\end{proof}

For a left symmetric algebra $V$ we write $\g(V)$ to denote 
its Lie algebra and $\cU(V)$ to denote the enveloping algebra of $\g(V)$.

The following proposition formalizes the fact that left symmetric algebras
are Lie theoretical analogs of classical braces. 
This phenomenon was already observed in~\cite{MR3465351} and~\cite{RumpClassical}. 

\begin{pro}
\label{pro:LSA}
Let $V$ be a left symmetric algebra. Then $\cU(V)$ is a Hopf brace.
\end{pro}

\begin{proof}
There exists a bijective $1$-cocycle $\g\to V$, 
where $V$ is considered as a $\g(V)$-module with action given by left multiplication. By Lemma~\ref{lema:1-cocycle-from-Lie-to-enveloping}, this $1$-cocycle admits
an extension to a bijective $1$-cocycle $\cU(V)\to S(V)$, where 
$S(V)$ is the symmetric algebra of $V$. Then the claim follows from Theorem~\ref{thm:main}.
\end{proof}

\begin{example}
Let $V$ be a 2-dimensional vector space with basis $x,y$. Fix $\alpha\in k$. Then the bilinear map $V\times V\to V$ such that
\begin{align*}
    x\cdot x &= 0, & x\cdot y &= 0, & y \cdot x &= x, & y\cdot y &= \alpha y,
\end{align*}
makes it a left symmetric algebra. The associated Lie algebra $\g$ is the solvable 2-dimensional Lie algebra such that $[x,y]=x$. Thus it has an associated bijective $1$-cocycle $\g\to V$ which extends to a bijective $1$-cocycle $\pi:\cU(\g)\to S(V)$ given by
\begin{align*}
    \pi(x^m y^n) &= \sum_{j=1}^n \tau_n(j) \alpha^{n-j} \, x^m y^{j}, & & m,n\in\N_0.
\end{align*}
Here, $\tau_n(j)$, $1\le j\le n$, is defined recursively as follows:
\begin{align}\label{eq:defn-tau}
\tau_n(1)&=\tau_n(n)=1, & \tau_{n+1}(j) &= \tau_n(j-1) +j \tau_n(j), \quad 2\le j\le n.
\end{align}
The proof is direct using \eqref{eq:extension-1-cocycle}. Notice that $\tau_n(1)=1$, $\tau_n(2)=2^{n-1}-1$ for all $n$.
\end{example}

\begin{exa}
This example is based on~\cite{MR1296585}. 
Let $k$ be a field of characteristic 3, $\alpha\in k^{\times}$. Let $V$ be a 3-dimensional vector space with basis $x,y,z$. Then the bilinear map $V\times V\to V$ such that
\begin{align*}
    x\cdot x &= 0, & y\cdot x &= (1-\alpha^{-1})z, & z\cdot x &= (\alpha+1)x, \\
    x\cdot y &= -(1-\alpha^{-1})z, & y\cdot y &= 0, & z\cdot y &= (\alpha-1)y, \\
    x\cdot z &= \alpha x, & y\cdot z &= \alpha z, & z\cdot z &= \alpha z,
\end{align*}
makes it a left symmetric algebra. The associated Lie algebra is $\sld$: we fix the classical basis $e,f,h$, where $[e,f]=h$, $[h,e]=2e$, $[h,f]=-2f$. Thus it has an associated bijective $1$-cocycle $\sld\to V$ such that $e\mapsto x$, $f\mapsto y$, $h\mapsto z$,
which extends to a bijective $1$-cocycle $\pi:\cU(\sld)\to S(V)$. We compute it explicitly.

By direct computation we have that
\begin{align*}
    \pi(e^n) &= x^n, & \pi(f^n)&= y^n, & \pi(h^n) &= \sum_{j=1}^n \tau_n(j)\alpha^{n-j} z^j, & n &\in\N,
\end{align*}
for $\tau_n(j)$ as in \eqref{eq:defn-tau}. 
As the action is by derivations, and $f^3\rightharpoonup x=0$,
\begin{align*}
    f^j \rightharpoonup x^a &= \sum_{t=0}^{\lfloor j/2\rfloor} \alpha^t (1-\alpha^{-1})^{j-t} \binom{a}{t} \binom{a-t}{j-2t} z^{j-2t}y^t x^{a-j+t}, & j,a &\in \N.
\end{align*}
Let $a,b\in\N_0$. We compute, using \eqref{eq:cocycle-condition},
\begin{multline*}
    \pi(f^be^a) = \sum_{j=0}^b \binom{b}{j} \pi(f^{b-j}) f^j \rightharpoonup \pi(e^a) \\
    = \sum_{j=0}^b \sum_{t=0}^{\lfloor j/2\rfloor} \alpha^t (1-\alpha^{-1})^{j-t} \binom{b}{j}    \binom{a}{t} \binom{a-t}{j-2t} z^{j-2t}y^{b-j+t} x^{a-j+t}.
\end{multline*}
By \eqref{eq:cocycle-condition} again, for all $a,b,c\in\N_0$ we have that
\begin{multline*}
    \pi(h^cf^be^a)
    = \sum_{k=0}^c \sum_{j=0}^b \sum_{s=1}^k  \sum_{t=0}^{\lfloor j/2\rfloor}
    \binom{c}{k} \binom{b}{j} \binom{a}{t} \binom{a-t}{j-2t} \\
    \tau_k(s)\alpha^{k-s+t}(1-\alpha^{-1})^{j-t}
    \big( (b-a+j)\alpha+b-a \big)^{c-k}
    z^{j-2t+s} y^{b-j+t} x^{a-j+t}.
\end{multline*}
\end{exa}

\section{Affine braces and commutative Hopf co-braces}
\label{schemes}

We will denote by $\operatorname{Alg}_k$  the category of all commutative $k$-algebras.  Each commutative $k$-algebra $A$ defines a contravariant functor $$h^A:=\Hom_{\operatorname{Alg}_k}(A,-):\operatorname{Alg}_k\to \operatorname{Set}.$$ A functor $\mathcal{A}:\operatorname{Alg}_k\to \operatorname{Set}$ is said to be \textit{representable} if it is isomorphic to $h^A$ for some commutative $k$-algebra. A representable functor $\mathcal{A}:\operatorname{Alg}_k\to \operatorname{Set}$ is called an \textit{affine scheme}.

\begin{defn}
A brace functor is a functor  $\mathcal{A}:\operatorname{Alg}_k\to \operatorname{Set}$, together with  natural transformations $m,m':\mathcal{A}\times \mathcal{A}\to \mathcal{A}$ such that, for all commutative $k$-algebra $R$
\begin{align*}
    m(R):\mathcal{A}(R)\times \mathcal{A}(R) &\to \mathcal{A}(R), &&& m'(R):\mathcal{A}(R)\times \mathcal{A}(R) &\to \mathcal{A}(R),\\ (a,b) &\mapsto ab, &&&     (a,b) &\mapsto a\circ b,
\end{align*}
define a  brace structure on $\mathcal{A}(R)$. A  brace functor $(\mathcal{A},m,m')$ is called an \textit{affine} brace if $\mathcal{A}$ is an affine scheme. 
\end{defn}

\begin{rem}
If $(\mathcal{A},m,m')$ is a brace functor, then it defines a functor from $\operatorname{Alg}_k$ to the category of braces.
\end{rem}
\begin{example}
A group functor $G$ is called a semidirect product of the subgroup functors $N$ and $Q$ if $N$ is normal and the map $N(R)\times Q(R)\to G(R), (n,q)\mapsto nq$ is a bijection of sets for all commutative $k$-algebras $R$. 
Let $G$ be a semidirect product of $N$ and $Q$. Then the functor $N\times Q$ is a brace functor with products on $N(R)\times Q(R)$
\begin{align*}
        &(n,q)(n',q')=(nn',qq'),&&(n,q)\circ(n'q')=(n(qnq^{-1}),qq'), 
    \end{align*}
$n,n'\in N(R),\,q,q'\in Q(R)$. If $N$ and $Q$ are affine, this construction defines an affine brace.
\end{example}

Let $\mathbb{A}^1$ be the forgetful functor, $$\mathbb{A}^1:\operatorname{Alg}_k\to \operatorname{Set},\  \  \   (R,+,\cdot,1)\mapsto  R.$$

Let $\mathcal{A}$ be an affine scheme. The set of natural transformations  $$\mathcal{O}(\mathcal{A}):=\operatorname{Nat}(\mathcal{A},\mathbb{A}^1)$$ has a  commutative $k$-algebra structure and there is a canonical natural isomorphism  $\alpha:\mathcal{A}\to h^{\mathcal{O}(\mathcal{A})}$ (see \cite{waterhouse2012introduction} for details). The $k$-algebra $\mathcal{O}(\mathcal{A})$ is called the (canonical) coordinate ring of $\mathcal{A}$.

\begin{defn}
	\label{def:cobrace}
	Let $(A,m,1)$ be an algebra. A \emph{Hopf co-brace structure} over $A$ consist of the following data:
	\begin{enumerate}
	    \item a Hopf algebra structure $(A,m,1,\Delta_\cdot,\epsilon_\cdot,S)$ and 
	    \item a Hopf algebra structure $(A,m,1,\Delta_\circ,\epsilon_\circ,T)$ 
	\end{enumerate}
	satisfying the following compatibility:
	\begin{align}
		\label{eq:cobrace}
		&a_{1_\circ}\otimes (a_{2_\circ})_{1_\cdot}\otimes (a_{2_\circ})_{2_\cdot}=(a_{1\cdot})_{1_\circ}S(a_{2\cdot})(a_{3\cdot})_{1_\circ}\otimes (a_{1\cdot})_{2_{\circ}}\otimes(a_{3\cdot})_{2_\circ}
	\end{align}
	for all $a,b,c\in A$.
\end{defn}

\begin{exa}
If $A$ is a cocommutative Hopf brace, then $A^{\circ}$ (the finite dual of $A$) is a commutative Hopf co-brace. 
\end{exa}
The Yoneda lemma implies that for any two commutative algebras $A, B$, there is a natural correspondence between elements of $\Hom_{\operatorname{Alg}}(A,B)$ and $\operatorname{Nat}(h^B,h^A)$, the natural transformation from $h^B$ to $h^A$.

If $H$ is a commutative Hopf algebra, then $h^H$ is an affine group, i.e. for every $k$-algebra $R$, $h^H(R)$ is a group with group structure given by the convolution product,
\begin{align*}
   m_{\Delta}: h^H(R)\times h^H(R)&\to h^H(R)\\
    (f_1,f_2) &\mapsto (f1*f_2)(h):=f_1(h_1)f_2(h_2).
\end{align*}
Conversely, for every affine group $G$ the coordinate ring $\mathcal{O}(G)$ has a Hopf algebra structure defined via Yoneda Lemma such that $G\cong h^{\mathcal{O}(G)}$ as group functors. 
\begin{pro}
If $A$ is a commutative Hopf co-brace then $h^A$ is an affine brace. Conversely, if $\mathcal{A}$ is an affine brace $\mathcal{O}(\mathcal{A})$ is  a commutative Hopf co-brace. This correspondence defines a canonical  contravariant equivalence of categories.
\end{pro}
\begin{proof}
The correspondence follows in the same lines as the correspondence between affine groups and commutative Hopf algebras, see \cite[Section 1.4]{waterhouse2012introduction}. We will use the fact that the category of affine schemes is equivalent to the opposite category of commutative $k$-algebras via  Yoneda Lemma and the canonical coordinate algebra.

Let $(\mathcal{A},m,m')$ be an affine brace. Then,
\[\mathcal{A}\times \mathcal{A} \simeq h^{\mathcal{O}(\mathcal{A})}\times h^{\mathcal{O}(\mathcal{A})}\simeq h^{\mathcal{O}(\mathcal{A})\otimes \mathcal{O}(\mathcal{A})},\]
and the natural transformations $m$ and $m'$ correspond to comultiplications $\Delta_\cdot, \Delta_{\circ}: \mathcal{O}(\mathcal{A})\to \mathcal{O}(\mathcal{A})\otimes \mathcal{O}(\mathcal{A})$, such that $\mathcal{O}(\mathcal{A})$ is a commutative Hopf co-brace. In fact, by \cite[Section 1.4]{waterhouse2012introduction},  $(\mathcal{O}(\mathcal{A}),\Delta_\cdot)$ and $(\mathcal{O}(\mathcal{A}), \Delta_\circ)$ are Hopf algebras, we only need to check the equation \eqref{eq:cobrace}. The  (skew) brace condition say that the diagram of natural transformations

\begin{equation}\label{comm brace}
\begin{tikzcd}
\mathcal{A}^{\times 3} \ar{rrrr}{\mathcal{F}}  \ar{dd}{\id_{\mathcal{A}}\times m_\cdot} &&&& \mathcal{A}^{\times 5} \ar{dd}{m_\circ \times\id_{\mathcal{A}}\times m_\circ}\\\\
\mathcal{A}^{\times 2}\ar{rr}{m_\circ} && \mathcal{A}&& \mathcal{A}^{\times 3} \ar{ll}{m_\cdot^{(2)}}
\end{tikzcd}
\end{equation}commutes, where $\mathcal{F}(R)(a,b,c)= (a,b,a^{-1},a,c)$. Since, the category of affine schemes is  equivalent to the opposite of category of commutative $k$-algebra, diagram \eqref{comm brace} implies the commutativity of the diagram

\begin{equation}\label{comm Hopf}
\begin{tikzcd}
\mathcal{O}(\mathcal{A})^{\otimes 3}   &&&& \ar{llll}{\widehat{\mathcal{F}}} \mathcal{O}(\mathcal{A})^{\otimes 5}   \\\\
\ar{uu}{\id_{\mathcal{O}(\mathcal{A})}\otimes \Delta_\cdot} \mathcal{O}(\mathcal{A})^{\otimes 2} && \ar{ll}{\Delta_\circ} \mathcal{O}(\mathcal{A}) \ar{rr}{\Delta_\cdot^{(2)}} && \mathcal{O}(\mathcal{A})^{\otimes 3}  \ar{uu}{\Delta_\circ \times\id_{\mathcal{O}(\mathcal{A})}\times \Delta_\circ}
\end{tikzcd}
\end{equation}
where $\widehat{\mathcal{F}}(a\otimes b\otimes c\otimes d\otimes e)=a\circ b\otimes S(c)\otimes (d\circ e)$. This diagram is exactly \eqref{eq:cobrace}.

Conversely, let $A$ be a commutative Hopf co-brace with comultiplications $\Delta\cdot, \Delta_{\circ}$. It follows directly that \eqref{eq:cobrace} implies that for every commutative $k$-algebra the maps \begin{align*}
m_{\cdot}: h^A(R)\times h^A(R) &\to h^A(R)\\
 (f_1,f_2) &\mapsto f_1\cdot f_2 := (f_1,f_2)\circ \Delta_\cdot,\\
 m_{\circ}: h^A(R)\times h^A(R) &\to h^A(R)\\
 (f_1,f_2) &\mapsto f_1\cdot f_2 := (f_1,f_2)\circ \Delta_\circ.
\end{align*} define a brace structure on $h^A(R)$.
\end{proof}

\section*{Acknowledgements}
The work of Angiono was partially supported by CONICET, Secyt (UNC), PICT-2013-1414 and MathAmSud project GR2HOPF.
Galindo was partially supported by the FAPA funds from
Vicerrector\'ia de Investigaciones de la Universidad de los Andes.
The work of Vendramin is partially supported by CONICET, PICT-2014-1376, MATH-AmSud, and ICTP.

Part of this work was done during a visit of Angiono and Vendramin to the Universidad de los Andes.
They would like to thank Universidad de los Andes for the warm hospitality and support.

\bibliographystyle{abbrv}
\bibliography{refs}

\def\cprime{$'$}
\begin{thebibliography}{10}

\bibitem{MR594432}
E.~Abe.
\newblock {\em Hopf algebras}, volume~74 of {\em Cambridge Tracts in
  Mathematics}.
\newblock Cambridge University Press, Cambridge, 1980.
\newblock Translated from the Japanese by Hisae Kinoshita and Hiroko Tanaka.

\bibitem{BachillerP3}
D.~Bachiller.
\newblock Classification of braces of order {$p^3$}.
\newblock {\em J. Pure Appl. Algebra}, 219(8):3568--3603, 2015.

\bibitem{MR3465351}
D.~Bachiller.
\newblock Counterexample to a conjecture about braces.
\newblock {\em J. Algebra}, 453:160--176, 2016.

\bibitem{MR2503192}
C.~Bai.
\newblock Bijective 1-cocycles and classification of 3-dimensional
  left-symmetric algebras.
\newblock {\em Comm. Algebra}, 37(3):1016--1057, 2009.

\bibitem{MR1296585}
D.~Burde.
\newblock Left-symmetric structures on simple modular {L}ie algebras.
\newblock {\em J. Algebra}, 169(1):112--138, 1994.

\bibitem{MR2233854}
D.~Burde.
\newblock Left-symmetric algebras, or pre-{L}ie algebras in geometry and
  physics.
\newblock {\em Cent. Eur. J. Math.}, 4(3):323--357 (electronic), 2006.

\bibitem{MR3177933}
F.~Ced{\'o}, E.~Jespers, and J.~Okni{\'n}ski.
\newblock Braces and the {Y}ang-{B}axter equation.
\newblock {\em Comm. Math. Phys.}, 327(1):101--116, 2014.

\bibitem{MR1183474}
V.~G. Drinfel{\cprime}d.
\newblock On some unsolved problems in quantum group theory.
\newblock In {\em Quantum groups ({L}eningrad, 1990)}, volume 1510 of {\em
  Lecture Notes in Math.}, pages 1--8. Springer, Berlin, 1992.

\bibitem{MR1722951}
P.~Etingof, T.~Schedler, and A.~Soloviev.
\newblock Set-theoretical solutions to the quantum {Y}ang-{B}axter equation.
\newblock {\em Duke Math. J.}, 100(2):169--209, 1999.

\bibitem{GI15}
T.~{Gateva-Ivanova}.
\newblock {Set-theoretic solutions of the Yang-Baxter equation, Braces, and
  Symmetric groups}.
\newblock {\em arXiv:1507.02602}, 2015.

\bibitem{MR1637256}
T.~Gateva-Ivanova and M.~Van~den Bergh.
\newblock Semigroups of {$I$}-type.
\newblock {\em J. Algebra}, 206(1):97--112, 1998.

\bibitem{MR585730}
J.~A. Green, W.~D. Nichols, and E.~J. Taft.
\newblock Left {H}opf algebras.
\newblock {\em J. Algebra}, 65(2):399--411, 1980.

\bibitem{GV}
L.~{Guarnieri} and L.~{Vendramin}.
\newblock {Skew braces and the Yang-Baxter equation}.
\newblock {\em Accepted for publication in Math. Comp. arXiv:1511.03171}, 2015.

\bibitem{MR1321145}
C.~Kassel.
\newblock {\em Quantum groups}, volume 155 of {\em Graduate Texts in
  Mathematics}.
\newblock Springer-Verlag, New York, 1995.

\bibitem{MR868976}
H.~Kim.
\newblock Complete left-invariant affine structures on nilpotent {L}ie groups.
\newblock {\em J. Differential Geom.}, 24(3):373--394, 1986.

\bibitem{MR1769723}
J.-H. Lu, M.~Yan, and Y.-C. Zhu.
\newblock On the set-theoretical {Y}ang-{B}axter equation.
\newblock {\em Duke Math. J.}, 104(1):1--18, 2000.

\bibitem{MR654637}
A.~Medina~Perea.
\newblock Flat left-invariant connections adapted to the automorphism structure
  of a {L}ie group.
\newblock {\em J. Differential Geom.}, 16(3):445--474 (1982), 1981.

\bibitem{MR2278047}
W.~Rump.
\newblock Braces, radical rings, and the quantum {Y}ang-{B}axter equation.
\newblock {\em J. Algebra}, 307(1):153--170, 2007.

\bibitem{RumpClassical}
W.~Rump.
\newblock The brace of a classical group.
\newblock {\em Note Mat.}, 34(1):115--144, 2014.

\bibitem{2015arXiv150900420S}
A.~{Smoktunowicz}.
\newblock {On Engel groups, nilpotent groups, rings, braces and the Yang-Baxter
  equation}.
\newblock {\em arXiv:1509.00420}, 2015.

\bibitem{MR1809284}
A.~Soloviev.
\newblock Non-unitary set-theoretical solutions to the quantum {Y}ang-{B}axter
  equation.
\newblock {\em Math. Res. Lett.}, 7(5-6):577--596, 2000.

\bibitem{waterhouse2012introduction}
W.~C. Waterhouse.
\newblock {\em Introduction to affine group schemes}, volume~66.
\newblock Springer Science \& Business Media, 2012.

\end{thebibliography}

\end{document}